\documentclass[10pt,english]{smfart}

\usepackage[T1]{fontenc}
\usepackage[francais]{babel}

\usepackage{tikz}
\usetikzlibrary{matrix,positioning}
\usepackage{amssymb,url,xspace,smfthm,amsmath,amsfonts,enumitem}

\DeclareFontFamily{U}{wncy}{}
\DeclareFontShape{U}{wncy}{m}{n}{
<5>wncyr5
<6>wncyr6
<7>wncyr7
<8>wncyr8
<9>wncyr9
<10>wncyr10
<11>wncyr10
<12>wncyr6
<14>wncyr7
<17>wncyr8
<20>wncyr10
<25>wncyr10}{}
\DeclareMathAlphabet{\cyr}{U}{wncy}{m}{n}

\newcommand{\BibTeX}{{\scshape Bib}\kern-.08em\TeX}
\newcommand{\T}{\S\kern .15em\relax }
\newcommand{\AMS}{$\mathcal{A}$\kern-.1667em\lower.5ex\hbox
        {$\mathcal{M}$}\kern-.125em$\mathcal{S}$}

\usepackage{tikz}
\usetikzlibrary{matrix}

\begin{document}
\title[Insufficiency of the \'etale Brauer--Manin obstruction]{Insufficiency of the \'etale Brauer--Manin obstruction: towards a simply connected example}
\date{}
\author{Arne Smeets}
\address{Imperial College London, Department of Mathematics, London SW7 2AZ, United Kingdom \emph{and} University of Leuven, Departement Wiskunde, Celestijnenlaan 200B, 3001 Heverlee, Belgium} 
\email{arnesmeets@gmail.com}

\begin{abstract} Since Poonen's construction of a variety $X$ defined over a number field $k$ for which $X(k)$ is empty and the \'etale Brauer--Manin set $X(\mathbf{A}_k)^\text{Br,\'et}$ is not, several other examples of smooth, projective varieties have been found for which the \'etale Brauer--Manin obstruction does not explain the failure of the Hasse principle. All known examples are constructed using ``Poonen's trick'', i.e. they have the distinctive feature of being fibrations over a higher genus curve; in particular, their Albanese variety is non-trivial. 

In this paper, we construct examples for which the Albanese variety is trivial. The new geometric ingredient in our construction is the appearance of Beauville surfaces. Assuming the $abc$ conjecture and using geometric work of Campana on orbifolds, we also prove the existence of an example which is simply connected.
\end{abstract}
\maketitle
\tableofcontents
\section{Introduction}

Let $k$ be a number field. Let $\Omega_k$ be the set of places of $k$; for each $v \in \Omega_k$, denote by $k_v$ the corresponding completion. Let $\mathbf{A}_k$ be the ad\`ele ring of $k$. Let $X$ be a smooth, projective, geometrically integral $k$-variety and let $\mathrm{Br}(X) = H^2_{\text{\'et}}(X,\mathbf{G}_m)$ be its Brauer group. Manin defined the so-called \emph{Brauer--Manin pairing} between $X(\mathbf{A}_k)$, the set of adelic points of $X$, and the Brauer group of $X$; it is given by $$X(\mathbf{A}_k) \times \mathrm{Br}(X) \to \mathbf{Q}/\mathbf{Z}: ((P_v)_{v \in \Omega_k},\alpha) \mapsto \sum_{v \in \Omega_k} j_v(\alpha(P_v)),$$ where $j_v: \mathrm{Br}(k_v) \to \mathbf{Q}/\mathbf{Z}$ is the local invariant (as defined in local class field theory). Denote by $X(\mathbf{A}_k)^{\text{Br}}$ the left kernel of this pairing; by the global reciprocity law in class field theory, the diagonal embedding $X(k) \to X(\mathbf{A}_k)$ induces an inclusion $X(k) \subseteq X(\mathbf{A}_k)^{\text{Br}}$. 

If $X(\mathbf{A}_k)^{\text{Br}} = \emptyset$ (and hence $X(k) = \emptyset$) whereas $X(\mathbf{A}_k) \neq \emptyset$, one says that  there is a Brauer--Manin obstruction to the Hasse principle. Manin proved in his 1970 ICM talk that the Brauer--Manin obstruction is the only obstruction to the Hasse principle for genus $1$ curves, assuming the finiteness of the Shafarevich--Tate group of the Jacobian. Many other positive statements of this type are now known; e.g. in their seminal work on Ch\^atelet surfaces, Colliot-Th\'el\`ene, Sansuc and Swinnerton-Dyer proved that the Brauer--Manin obstruction is the only obstruction to the Hasse principle and weak approximation for these varieties. 

In 1999, Skorobogatov was the first to construct an example of a smooth, projective variety for which the failure of the Hasse principle is \emph{not} explained by a Brauer--Manin obstruction (see \cite{Sk}); his example is a bi-elliptic surface over $\mathbf{Q}$. In his paper \cite[\S 3]{Sk}, he defined a refinement of the Brauer--Manin obstruction, the \emph{\'etale Brauer--Manin obstruction}, which explains the failure of the Hasse principle in his case. Subsequently, in \cite{poonen}, Poonen constructed the first examples of varieties (over any number field) for which the Hasse principle fails, and yet the \emph{\'etale} Brauer--Manin set, which is contained in the classical Brauer--Manin set, is non-empty. His examples are pencils of Ch\^atelet surfaces parametrized by a curve of genus at least $1$. Very recently, Harpaz and Skorobogatov \cite{hs} were the first to construct \emph{surfaces} (over any number field) for which the failure of the Hasse principle is not explained by an \'etale Brauer--Manin obstruction; these are pencils of curves parametrized by a curve with exactly one rational point (hence of genus at least $1$). Finally, in their paper \cite{cps}, Colliot-Th\'el\`ene, P\'al and Skorobogatov even managed to construct various families of quadrics which illustrate the insufficiency of the \'etale Brauer--Manin obstruction;  each of these examples maps again (dominantly) to a higher genus curve.

Therefore a distinctive feature common to all known examples of varieties for which the (\'etale) Brauer--Manin obstruction does not suffice to explain the failure of the Hasse principle, is the fact that they are fibrations over a ``base curve'' of genus at least $1$; in particular, their Albanese variety is non-trivial. It is natural to ask whether there exist examples for which the Albanese variety is trivial; to the best of the author's knowledge, no such examples have been found before. Such examples should exist if one believes Lang's conjectures: Sarnak and Wang \cite{sw} proved that the existence of simply connected varieties over number fields, with local points everywhere and without any rational points, would follow from these conjectures. These varieties are (smooth) hypersurfaces in $\mathbf{P}^4_{\mathbf{Q}}$, for which the Brauer--Manin obstruction and its refinements are known to vanish automatically. Constructing a simply connected example is one of the biggest challenges in the field.

In this paper, we will construct examples of smooth, projective varieties with trivial Albanese variety, such that the failure of the Hasse principle is not accounted for by the \'etale Brauer--Manin obstruction. The key new ingredient of our construction is very geometric: in \S 2, we construct (over an arbitrary number field) an explicit Beauville surface needed for our examples. In \S 3, we will produce these examples as families of quadrics or families of Ch\^atelet surfaces over such a Beauville surface. Finally, in \S 4 we will show, using geometric work of Campana, that the $abc$ conjecture implies the existence of a simply connected example.

\paragraph*{More notation and conventions.} See the beginning of this introduction for the definition of the Brauer--Manin obstruction; for the definition of the \emph{\'etale Brauer--Manin obstruction}, we refer to \cite[\S 3]{Sk}. 

For a smooth, projective and geometrically integral surface $S$ over an arbitrary field $k$, we use the classical notation $q(S) = \dim H^1(S,\mathcal{O}_S)$ for the irregularity, i.e. the dimension of the Albanese variety, and $p_g(S) = \dim H^2(S,\mathcal{O}_S)$ for the geometric genus. 

\paragraph*{Acknowledgements.} It is a great pleasure to thank my advisor, J.-L. Colliot-Th\'el\`ene, for proposing this problem and for his help and encouragement. I would also like to thank W. Castryck, F. Catanese, A. Javanpeykar and O. Wittenberg for useful suggestions, and D. Abramovich and P. Bruin for discussions on the subject of the appendix. The author was supported by a PhD fellowship of the Research Foundation Flanders (FWO Vlaanderen) and by ERC grant MOTMELSUM.
\section{An explicit Beauville surface}

Let $k$ be an infinite field, not of characteristic $2$. We will construct an explicit smooth, projective, geometrically integral surface $S$ of general type over $k$, with a rational point, and only finitely rational points if the field $k$ is finitely generated over $\mathbf{Q}$. Our surface belongs to the class of Beauville surfaces. These are ``fake quadrics'': surfaces of general type having the same numerical invariants as $\mathbf{P}^1 \times \mathbf{P}^1$.

Our construction is inspired by work of Bauer and Catanese on such surfaces (over $\mathbf{C}$) which appeared in \cite{bc}. We carry out the construction over any field of characteristic not $2$ in a very elementary way; this allows us to deal with some requirements of arithmetic nature as well. We will construct this Beauville surface $S$ starting from two smooth, projective curves over $k$, a curve $\mathcal{C}$ of genus $5$ and a curve $\mathcal{D}$ of genus $3$, both equipped with an action of the group $G = (\mathbf{Z}/2)^3$ such that $\mathcal{C}/G \cong \mathbf{P}^1_k$ and $\mathcal{D}/G \cong \mathbf{P}^1_k$, and such that the diagonal action of $G$ on $\mathcal{C} \times \mathcal{D}$ is free; we will then take $S$ to be the quotient $(\mathcal{C} \times \mathcal{D})/G$.

Consider three smooth conics $C_i$ over $k$ (for $i \in \{1,2,3\}$), each of them equipped with a dominant morphism $\pi_{C_i}: C_i \to \mathbf{P}^1_k$ of degree $2$, ramified above two $k$-rational points $P_i$ and $P_i'$ of $\mathbf{P}^1_k$, such that the following conditions are satisfied: \begin{itemize} \item[$-$]  the six points $P_1$, $P_1'$, $P_2$, $P_2'$, $P_3$ and $P_3'$ are pairwise distinct; \item[$-$] the sets $\pi_{C_i}(C_i(k)) \subseteq \mathbf{P}^1(k)$ have a point in common, not one of the $P_i$ or $P_i'$. \end{itemize} It is clear that such a triple of conics exists. Consider the smooth curve $$\mathcal{C} = C_1 \times_{\mathbf{P}^1_k} C_2 \times_{\mathbf{P}^1_k} C_3.$$ We certainly have $\mathcal{C}(k) \neq \emptyset$. The genus $g_{\mathcal{C}}$ of $\mathcal{C}$ can be calculated using the Riemann-Hurwitz theorem: the resulting morphism $\pi_{\mathcal{C}}: \mathcal{C} \to \mathbf{P}^1_k$ has degree $8$. The ramification divisor on $\mathcal{C}$ consists of $24$ points, four of them above each of $P_1$, $P_1'$, $P_2$, $P_2'$, $P_3$ and $P_3'$, and each of them has ramification degree equal to $2$. We get $$2 - 2g_{\mathcal{C}} = 8 \cdot (2 - 2g_{\mathbf{P}_1^k}) - 24$$ and hence $g_{\mathcal{C}} = 5$.

Consider three smooth conics $D_i$ over $k$ (for $i \in \{1,2,3\}$), each of them again equipped with a morphism $\pi_{D_i}: D_i \to \mathbf{P}^1_k$ of degree $2$, ramified above two $k$-rational points $Q_i$ and $Q_i'$ of $\mathbf{P}^1_k$, such that the following conditions are satisfied: \begin{itemize} \item[$-$] the points $Q_1$, $Q_1'$, $Q_2$, $Q_2'$ and $Q_3$ are pairwise distinct, and $Q_3' = Q_2'$; \item[$-$] the sets $\pi_{D_i}(D_i(k)) \subseteq \mathbf{P}^1(k)$ have a point in common, not one of the $Q_i$ or $Q_i'$. \end{itemize} It is again clear that such a triple of conics exists. Consider the curve $$D = D_1 \times_{\mathbf{P}^1_k} D_2 \times_{\mathbf{P}^1_k} D_3,$$ which has two singular points lying above $Q_2' = Q_3'$. Let $\mathcal{D}$ be the normalization of $D$, then clearly $\mathcal{D}(k) \neq \emptyset$. The induced morphism $\pi_{\mathcal{D}}: \mathcal{D} \to \mathbf{P}^1_k$ has degree $8$. The ramification divisor on $\mathcal{D}$ now consists of $20$ points, four above each of $Q_1$, $Q_1'$, $Q_2$, $Q_2'$ and $Q_3$, of ramification degree equal to $2$. The Riemann-Hurwitz theorem gives $$2 - 2g_{\mathcal{D}} = 8 \cdot (2 - 2g_{\mathbf{P}_1^k}) - 20$$ and hence $g_{\mathcal{D}} = 3$. 

We will now define an action of $G = (\mathbf{Z}/2)^3$ on $\mathcal{C}$ and $\mathcal{D}$ such that the diagonal action on $\mathcal{C} \times \mathcal{D}$ is free, and moreover $\mathcal{C}/G = \mathbf{P}^1_k$ and $\mathcal{D}/G = \mathbf{P}^1_k$. The points on $\mathcal{C}$ and $\mathcal{D}$ having non-trivial stabilizer are precisely the points on the ramification divisors; we need to define a $G$-action on both $\mathcal{C}$ and $\mathcal{D}$ such that the set of stabilizers for the action on $\mathcal{C}$ is disjoint from the set of stabilizers for the action on $\mathcal{D}$. Let $G = \langle a, b, c \rangle$ be a presentation. On each of the conics $C_i$ and $D_i$ we have an obvious involution; let us call these involutions $\sigma_i$ and $\tau_i$ respectively. Let $a$ act on $\mathcal{C}$ as $(\sigma_1,\mathbf{1},\mathbf{1})$ and on $D$ as $(\tau_1,\mathbf{1},\tau_3)$; let $b$ act on $\mathcal{C}$ as $(\mathbf{1},\sigma_2,\mathbf{1})$ and on $D$ as $(\tau_1,\tau_2,\tau_3)$; let $c$ act on $\mathcal{C}$ as $(\mathbf{1},\mathbf{1},\sigma_3)$ and on $D$ as $(\tau_1,\tau_2,\mathbf{1})$. This defines $G$-actions on $\mathcal{C}$ and $D$. The $G$-action on $D$ extends to $\mathcal{D}$. It is easy to check that $\mathcal{C}/G \cong \mathbf{P}^1_k$, that $\mathcal{D}/G \cong \mathbf{P}^1_k$ and that the diagonal action of $G$ on $\mathcal{C} \times \mathcal{D}$ is free.

\begin{prop} \label{beauville} With the above notation, define $$S = (\mathcal{C} \times \mathcal{D})/G.$$ Then $S$ is a smooth, projective, geometrically integral surface with a rational point satisfying $q(S) = p_g(S) = 0$. The projections $\pi_{\mathcal{C}}: \mathcal{C} \times \mathcal{D} \to \mathcal{C}$ and $\pi_{\mathcal{D}}: \mathcal
{C} \times \mathcal{D} \to \mathcal{D}$ induce two morphisms $\varphi_\mathcal{C}: S \to \mathbf{P}^1_k$ and $\varphi_{\mathcal{D}}: S \to \mathbf{P}^1_k$ having six and five double fibres respectively. 

If the field $k$ is finitely generated over $\mathbf{Q}$, then the set $S(k)$ is finite. \end{prop}

\begin{proof} We defined the action of $G$ on $\mathcal{C}$ and $\mathcal{D}$ in such a way that the diagonal action of $G$ on the product $\mathcal{C} \times \mathcal{D}$ is free; this implies that the quotient $S$ is smooth. Since both $\mathcal{C}$ and $\mathcal{D}$ have a rational point, we immediately see that $S(k) \neq \emptyset$.

Let us calculate the basic invariants of the surface $S$; this is rather classical. We have $\chi(\mathcal{C} \times \mathcal{D}) =  (g_\mathcal{C} - 1)(g_\mathcal{D} - 1) = 8$ and also $\sharp G = 8$; therefore $\chi(S) = 1$ and hence $q(S) = p_g(S)$. Moreover, the irregularity $q(S)$ is equal to \begin{eqnarray*} H^0(S,\Omega^1_S) & = & \dim H^0(\mathcal{C} \times \mathcal{D}, \Omega^1_{\mathcal{C} \times \mathcal{D}})^G \\ & = & \dim H^0(\mathcal{C},\Omega^1_\mathcal{C})^G + \dim H^0(\mathcal{D},\Omega^1_\mathcal{D})^G \\ & = & \dim H^0(\mathcal{C}/G,\Omega^1_{\mathcal{C}/G})  + \dim H^0(\mathcal{D}/G,\Omega^1_{\mathcal{D}/G}) \\ & = & g_{\mathcal{C}/G} + g_{\mathcal{D}/G}.\end{eqnarray*}  Since $\mathcal{C}/G \cong \mathbf{P}^1_k$ and $\mathcal{D}/G \cong \mathbf{P}^1_k$, we get $q(S) = 0$ and therefore also $p_g(S) = 0$. 

The projection $\mathcal{C} \times \mathcal{D} \to \mathcal{C}$ induces a morphism $$\varphi_{\mathcal{C}}: S = (\mathcal{C} \times \mathcal{D})/G \to \mathcal{C}/G = \mathbf{P}^1_k$$ which has six branch points; these are just the points $P_i$ and $P_i'$ mentioned before.  Above each of these points, we get a double fibre (a quotient of $\mathcal{D}$). If $k$ is finitely generated over $\mathbf{Q}$, then it follows from from \cite[Corollary 2.2]{paucity} that $\varphi_{\mathcal{C}}(S(k))$ is finite; the key point in the argument is the fact that each point in $\varphi_{\mathcal{C}}(S(k))$ must be either one of the branch points, or the image of a rational point on one of only \emph{finitely} many double covers of $\mathbf{P}^1_k$ ramified over the points where $\varphi$ has a double fibre. These curves have genus at least $2$, and hence have only finitely rational points by Faltings' theorem (Mordell's conjecture). 

Similarly, the projection $\mathcal{C} \times \mathcal{D} \to \mathcal{D}$ induces a morphism $$\varphi_{\mathcal{D}}: S = (\mathcal{C} \times \mathcal{D})/G \to \mathcal{C}/G = \mathbf{P}^1_k$$ with five branch points (the $Q_i$ and $Q_i'$). Now \cite[Corollary 2.4]{paucity} implies that $\varphi_{\mathcal{D}}(S(k))$ is finite (the argument is a little bit more complicated in this case). Since both $\varphi_{\mathcal{C}}(S(k))$ and $\varphi_{\mathcal{D}}(S(k))$ are finite, it follows that $S(k)$ is finite as well. \end{proof}

This gives a very simple and explicit example of a surface of general type, with trivial Albanese variety, for which the number of rational points is finite even after any finite extension of the base field.

\section{Examples with trivial Albanese variety}

We will borrow some terminology from \cite[\S 2]{paucity}: a \emph{standard fibration} is a dominant and proper morphism of varieties, with smooth and geometrically integral generic fibre. 

The main idea of the next two propositions is the following: assume we are given a smooth, projective, geometrically integral variety $Y$ of dimension at least $2$ over a number field $k$, equipped with a standard fibration $f: Y \to \mathbf{P}^1_k$ such that $f(Y(k))$ is finite (but non-empty). If $k$ has at least one real place, then there is a very simple way of producing a quadric bundle $X \to Y$ of relative dimension $3$ such that $X(k) = \emptyset$ but $X(\mathbf{A}_k)^{\text{\'et,Br}} \neq \emptyset$. This is the content of Proposition \ref{quadrics}, inspired by the constructions used in \cite[\S 3]{cps}. If $k$ is arbitrary, one can instead produce a family of Ch\^atelet surfaces $X \to Y$ with the same properties. This is Proposition \ref{chatelet};  this slightly more complicated construction is inspired by Poonen's work \cite{poonen}.

\begin{prop} \label{quadrics} Let $k$ be a real number field. Let $Y$ be a smooth, projective, geometrically integral $k$-variety. Assume that there exists a standard fibration $f: Y \to \mathbf{P}^1_k$ such that $f(Y(k))$ is finite and non-empty. 

There exists a standard fibration in three-dimensional quadrics $g: Q \to \mathbf{P}^1_k$ such that the fibre product $X = Y \times_{\mathbf{P}^1_k} Q$ is a smooth, projective, geometrically integral $k$-variety for which $X(k) = \emptyset$ and $X(\mathbf{A}_k)^{\emph{\'et,Br}} \neq \emptyset$. \end{prop}

\begin{proof} Fix any point $M \in Y(k)$ and let $T = f(Y(k))$. Fix a real place $v_0$ of $k$. There exists a standard fibration in three-dimensional quadrics $g : Q \to \mathbf{P}^1_k$ satisfying the following properties (simultaneously): \begin{itemize} \item[$-$] $g$ has smooth fibres above the points where the fibre of $f$ is singular; \item[$-$] if $P \in T$, then the fibre $Q_P$ is smooth and has no $k_{v_0}$-points; \item[$-$] for each infinite place $v$ of $k$, there exists a point $A_v$ in the connected component of $Y(k_v)$ containing $M$ such that the fibre $Q_{f(A_v)}$ is smooth and \emph{does} have a $k_v$-point. \end{itemize} Constructing a fibration with these properties boils down to a basic interpolation problem.

We now claim that such a fibration does the job. It is clear that $X = Y \times_{\mathbf{P}^1_k} Q$ does not have any rational point: if $P \in T$, then the fibre $Q_P$ does not even have a $k_{v_0}$-point. Hence we only need to prove that $X(\mathbf{A}_k)^{\text{\'et,Br}} \neq \emptyset$, i.e. that there is no \'etale Brauer--Manin obstruction.

The fibre $X_M$ of the morphism $X \to Y$ is a smooth  three-dimensional quadric over $k$. Hence $X_M$ has a $k_v$-point for every \emph{finite} place $v \in \Omega_k$; for each of these places, choose any such point $Z_v$. Let us also choose a $k_v$-point $Z_v$ on $X_{A_v}$ for every \emph{infinite} place $v \in \Omega_k$; such a point exists since (by construction) $Q_{f(A_v)}(k_v) \neq \emptyset$ for every such $v$. This gives rise to an adelic point $Z = (Z_v)_{v \in \Omega_k} \in X(\mathbf{A}_k)$. We will now prove that $Z \in X(\mathbf{A}_k)^{\text{\'et,Br}}$. 

Let $G$ be a finite $k$-group scheme. Let $X' \to X$ be a $G$-torsor. By \cite[Proposition 3.3]{hs}, there exists a $G$-torsor $Y' \to Y$ such that $X \times_Y Y' \to X$ and $X' \to X$ are isomorphic as $G$-torsors. The fibre of $Y' \to Y$ at $M$ is a $k$-torsor, defined by a $1$-cocycle $\sigma$. By twisting both $X' \to X$ and $Y' \to Y$ by $\sigma$ and replacing $G$ by the twisted group $G^\sigma$ if necessary, we may assume, without loss of generality, that $Y'$ has a $k$-point $M'$ which maps to $M$.

Let $Y''$ be the irreducible component of $Y'$ which contains $M'$ and let $X''$ be the inverse image of $Y''$ in $X'$; this is a geometrically integral $k$-variety. For each finite place $v \in \Omega_k$, there exists a $k_v$-point $Z_v' \in X''(k_v)$ which maps to $Z_v$. Since $Y'' \to Y$ is finite and \'etale, we know that if $v \in \Omega_k$ is an infinite place, the image of the connected component of $Y''(k_v)$ containing $M'$ is exactly the connected component of $Y(k_v)$ containing $M$. For each such $v$, we can choose a $k_v$-point $A_v'$ in the connected component of $Y''(k_v)$ containing $M'$ which maps to $A_v$ in $Y(k_v)$. Let $Z_v' \in X''_{A'_v}(k_v)$ be a point which maps to $Z_v \in X_{A_v}(k_v)$. 

We have now constructed an adelic point $Z' = (Z_v')_{v \in \Omega_k} \in X''(\mathbf{A}_k)$ which maps to $Z$. The only thing left to prove is the fact that $Z'$ is orthogonal to the Brauer group $\mathrm{Br}(X')$; it is in fact sufficient to prove that it is orthogonal to $\mathrm{Br}(X'')$. 

Consider the image of $Z'$ in $Y''$; this is an adelic point $z' \in Y''(\mathbf{A}_k)$ with the property that $z_v' = M'$ when $v$ is finite, and $z_v' = A_v'$ when $v$ is infinite. Since $A_v'$ lies in the connected component of $Y''(k_v)$ containing $M'$ for every infinite place $v$, this means that $z'$ lies in the connected component of $Y''(\mathbf{A}_k)$ containing $M'$. Hence $z'$ is orthogonal to $\mathrm{Br}(Y'')$.

By \cite[Proposition 2.1.(c)]{cps}, the map $\mathrm{Br}(Y'') \to \mathrm{Br}(X'')$ induced by $X'' \to Y''$ is surjective. Since $z'$ is orthogonal to $\mathrm{Br}(Y'')$, it follows that $Z'$ is orthogonal to $\mathrm{Br}(X'')$. \end{proof}

With a bit more work, one can also produce a quadric bundle of relative dimension $2$ (instead of $3$) which does the job (for the purpose of Proposition \ref{quadrics}). Since we included Proposition \ref{quadrics} mainly for the simplicity of the construction compared to the one used in Proposition \ref{chatelet}, we will not go into details on this.

\begin{prop} \label{chatelet} Let $k$ be any number field. Let $Y$ be a smooth, projective, geometrically integral $k$-variety. Assume that there exists a standard fibration $f: Y \to \mathbf{P}^1_k$ such that $f(Y(k))$ is finite and non-empty. 

There exists a standard fibration in Ch\^atelet surfaces $g: C \to \mathbf{P}^1_k$ such that the fibre product $X = Y \times_{\mathbf{P}^1_k} C$ is a smooth, projective, geometrically integral $k$-variety for which $X(k) = \emptyset$ and $X(\mathbf{A}_k)^{\emph{\'et,Br}} \neq \emptyset$.  \end{prop} 

\begin{proof} Let $\gamma: \mathbf{P}^1_k \to \mathbf{P}^1_k$ be a morphism which maps the set $f(Y(k))$ to the point at infinity $\infty \in \mathbf{P}^1(k)$ and let $h = \gamma \circ f$. Let $p_1: \mathbf{P}^1_k \times \mathbf{P}^1_k \to \mathbf{P}^1_k$ be the first projection.
For a suitable integer $a > 0$, one can find a section $s \in \Gamma(\mathbf{P}^1_k \times \mathbf{P}^1_k,\mathcal{L}^{\otimes 2})$, where $\mathcal{L}$ is the line bundle $\mathcal{O}(a,2)$ on $\mathbf{P}^1_k \times \mathbf{P}^1_k$, such that the following properties are satisfied:

\begin{itemize}
\item[$-$] the subscheme of $\mathbf{P}^1_k \times \mathbf{P}^1_k$ given by $s = 0$ is a smooth, proper, geometrically integral $k$-curve $Z'$;
\item[$-$] the composition $Z' \hookrightarrow \mathbf{P}^1_k \times \mathbf{P}^1_k \stackrel{p_1}{\longrightarrow} \mathbf{P}^1_k$ is \'etale above the points where $h$ is not smooth;
\item[$-$] in the notation of \cite[\S 4]{poonen}, the equation $y^2 - az^2 = s$ (for some $a \in k^\times$) defines a conic bundle $C'$ over $\mathbf{P}^1_k \times \mathbf{P}^1_k$ such that its restriction to $\{\infty\} \times \mathbf{P}^1_k$ is a Ch\^atelet surface $S'$ with $S'(k) = \emptyset$ and $S'(\mathbf{A}_k) \neq \emptyset$; such a surface exists by \cite[Proposition 5.1]{poonen2}.
\end{itemize} The composition $$C' \to \mathbf{P}^1_k \times \mathbf{P}^1_k \stackrel{p_1}{\longrightarrow} \mathbf{P}^1_k$$ defines a family of Ch\^atelet surfaces over $\mathbf{P}^1_k$; we will then take $C \to \mathbf{P}^1_k$ to be the pullback of this family along $\gamma$. The following diagram, in which all squares are Cartesian, gives a summary of the situation: 

\begin{center}
\begin{tikzpicture}[auto]
\node (L1) {$X$};
\node (L2) [below= 0.9cm of L1] {$Y \times \mathbf{P}^1_k$};
\node (L3) [below= 0.9cm of L2] {$Y$};
\node (M1) [right= 1.8cm  of L1] {$C$};
\node (M2) at (M1 |- L2) {$\mathbf{P}^1_k \times \mathbf{P}^1_k$};
\node (M3) at (M2 |- L3) {$\mathbf{P}^1_k $};
\node (R1) [right= 1.8cm  of M1] {$C'$};
\node (R2) at (R1 |- M2) {$\mathbf{P}^1_k \times \mathbf{P}^1_k$};
\node (R3) at (R2 |- M3) {$\mathbf{P}^1_k$};
\draw[->] (L1) to node {} (M1);
\draw[->] (M1) to node {} (R1);
\draw[->] (L2) to node {\footnotesize $(f,1)$} (M2);
\draw[->] (M2) to node {\footnotesize  $(\gamma,1)$} (R2);
\draw[->] (L1) to node {} (L2);
\draw[->] (M1) to node {} (M2);
\draw[->] (R1) to node {} (R2);
\draw[->] (L2) to node {\footnotesize  $p_1$} (L3);
\draw[->] (M2) to node {\footnotesize  $p_1$} (M3);
\draw[->] (R2) to node { \footnotesize $p_1$} (R3);
\draw[->] (L3) to node { \footnotesize $f$} (M3); 
\draw[->] (M3) to node {\footnotesize $\gamma$} (R3);
\end{tikzpicture}
 \end{center}

Now the curve $$(\gamma,1)^\star Z' \hookrightarrow \mathbf{P}^1_k \times \mathbf{P}^1_k$$ is certainly smooth and projective; since $Z'$ is ample on $\mathbf{P}^1_k \times \mathbf{P}^1_k$ and $\gamma$ is finite, we have that $(\gamma,1)^\star Z'$ is ample on $\mathbf{P}^1_k \times \mathbf{P}^1_k$, hence geometrically connected by \cite[Corollary III.7.9]{hartshorne}. Therefore $(\gamma,1)^\star Z'$ is geometrically integral.
  
  Define $$Z = (h,1)^\star Z' = (\gamma,1)^\star Z' \times_{\mathbf{P}^1_k} Y.$$ The projection $Z \to (\gamma,1)^\star Z'$ is again a standard fibration, since the generic fibre remains geometrically integral after a finite base change. Since the base curve $(\gamma,1)^\star Z'$ is smooth, projective and geometrically integral, this easily implies that the subvariety $Z \hookrightarrow Y \times \mathbf{P}^1_k$ has the same properties, and hence that the total space $X$ of the fibration is smooth. 
  
  Now $Z$ is the locus where the conic bundle $X \to Y \times \mathbf{P}^1_k$ degenerates; since $Z$ is smooth, projective and geometrically integral, \cite[Proposition 2.1.(i)]{cps} implies that the induced map $\text{Br}(Y \times \mathbf{P}^1_k) \to \text{Br}(X)$ is an isomorphism. It follows that the map $\text{Br}(Y) \to \text{Br}(X)$ is an isomorphism as well. 
  
  Since all fibres of $X \to Y$ over a rational point of $Y$ are isomorphic to the same Ch\^atelet surface $S'$ (which fails the Hasse principle), we have $X(k) = \emptyset$. Proving that $X(\mathbf{A}_k)^{\text{\'et,Br}} \neq \emptyset$ can be done in a similar way as in the proof of Proposition \ref{quadrics}; the crucial facts needed for the proof are the surjectivity of $\text{Br}(Y) \to \text{Br}(X)$, and the fact that every torsor over $X$ under a finite group scheme over $k$ arises from such a torsor over $Y$. \end{proof}

\begin{rema} \normalfont \label{pi1} One can mimic the argument given by Poonen in \cite[Lemma 8.1]{poonen} to prove that $\pi_1^{\text{\'et}}(X) \cong \pi_1^{\text{\'et}}(Y)$ in Proposition \ref{chatelet}: the geometric fibres of the map $X \to Y \times \mathbf{P}^1_k$ have no nontrivial \'etale covers, hence the category of finite \'etale covers of $X$ is equivalent to the category of finite \'etale covers of $Y \times \mathbf{P}^1_k$. The same argument applies to $Y \times \mathbf{P}^1_k \to Y$. Hence the induced maps $$\pi_1^{\text{\'et}}(Y) \to \pi_1^{\text{\'et}}(Y \times \mathbf{P}^1_k) \to \pi_1^{\text{\'et}}(X)$$ are isomorphisms. 
\end{rema}

Before we can state and prove the main result of this paper, we need the following rigidity lemma, which applies to families of quadrics and, more generally, to families of unirational varieties, rationally connected varieties, $\dots$

\begin{lemm} \label{geometry} Let $k$ be a perfect field and let $f: Y \to X$ be a standard fibration. Assume that $X$ has trivial Albanese variety, and that the same is true for at least one smooth fibre of $f$. Then the Albanese variety of $Y$ is trivial as well. \end{lemm}

\begin{proof} We may and will assume, without loss of generality, that $k$ is separably closed. Let $A$ be an abelian variety over $k$ and let $g: Y \to A$ be any morphism; we need to prove that it is constant. Consider the morphism $$(f,g): Y \to X \times_k A$$ and let $Z$ be the image of $Y$ under this morphism. This is an irreducible subvariety of $X \times_k A$ which is proper over $X$. Since a smooth fibre of $f$ has trivial Albanese variety, there exists a point $x_0 \in X$ such that the fibre $Z_{x_0}$ of $p: Z \to X$ is zero-dimensional. By Chevalley's theorem on the semicontinuity of the dimensions of the fibres of a morphism \cite[Ex. 4.3.22]{hartshorne}, $p$ is generically finite. Since its generic fibre is geometrically integral, $p$ is birational. Since $X$ has trivial Albanese variety, so has $Z$; this easily implies the result.  \end{proof}

We can now state our main result:

\begin{theo} Let $k$ be a number field. There exists a smooth, projective, geometrically integral fourfold $X$ over $k$ such that $X(k) = \emptyset$, $X(\mathbf{A}_k)^{\emph{\'et,Br}} \neq \emptyset$ and $\emph{Alb}(X)$ is trivial. \end{theo}

\begin{proof} Apply Proposition \ref{chatelet} with $Y$ equal to a Beauville surface of the type constructed in \S 2, and $f$ equal to either one of the maps $\varphi_{\mathcal{C}}$ and $\varphi_{\mathcal{D}}$; then $f(Y(k))$ is clearly finite. The triviality of the Albanese variety of $X = Y \times_{\mathbf{P}^1_k} C$ follows immediately from Lemma \ref{geometry}.  \end{proof}

\section{Simply connected examples and the $abc$ conjecture}

One of the big challenges in the field is the construction of a smooth, projective and geometrically simply connected variety over a number field such that the failure of the Hasse principle is not accounted for by the \'etale Brauer--Manin obstruction. Sarnak and Wang \cite{sw} proved that there exist hypersurfaces of degree $1130$ in $\mathbf{P}^4_{\mathbf{Q}}$ having this property, assuming Lang's conjectures. Another conditional proof of the existence of a (geometrically) simply connected example can be found in Poonen's paper \cite{poonenshort}. In this paragraph, we will prove that the existence of such a variety follows from the $abc$ conjecture, using Campana's work on orbifolds \cite{campana}.

If a standard fibration $f: X \to \mathbf{P}^1_k$ has at least two double fibres, then $X$ is not simply connected by \cite[Proposition 1.1]{paucity}; if $f$ has at least four double fibres, then $\pi_1^{\text{\'et}}(X)$ is infinite. Indeed, if $C \to \mathbf{P}^1_k$ is a double cover ramified over the points where the fibre of $f$ is a double, then the morphism $X \times_{\mathbf{P}^1_k} C \to X$ is an \'etale double cover. Since $X \times_{\mathbf{P}^1_k} C$ maps dominantly to a higher genus curve, its fundamental group is infinite, and hence so is the fundamental group of $X$. This suggests that one cannot keep the fundamental group of $X$ very small if one wants to use multiple fibres (in the classical sense) to control the image of $f(X(k))$. 

Campana used an alternative notion of multiplicity \cite[\S 1.2.2]{campana} on which he based his own version of the theory of orbifolds. This allows one to circumvent the kind of ``topological obstructions'' encountered above to the existence of multiple fibres, in the non-classical sense. The price one has to pay is that instead of using Faltings' theorem as in the proof of \cite[Corollary 2.2]{paucity}, one now needs to assume an orbifold version of the Mordell conjecture; as we will see in the appendix, this is actually a consequence of the $abc$ conjecture.

Note that for a smooth, projective, geometrically integral and simply connected variety defined over a number field, the \'etale Brauer--Manin set coincides with the classical Brauer--Manin set.

\begin{theo} Assume that the $abc$ conjecture is true. There exist a number field $k$ and a smooth, projective, geometrically integral and geometrically simply connected fourfold $X$ over $k$ such that $X(k) = \emptyset$ and $X(\mathbf{A}_k)^{\emph{Br}} \neq \emptyset$. \end{theo}

\begin{proof} The geometric ingredient is a construction of Campana done in \cite[\S 5]{campana}. He produced a standard fibration $f: S \to \mathbf{P}^1_k$, where $S$ is a smooth, projective, geometrically integral surface of general type which is simply connected, without multiple fibres in the classical sense, which can be defined over a number field $k$. By \cite[Proposition 4.3]{campana}, the image $f(S(k))$ of the set of rational points on $S$ is contained in the set of orbifold rational points $(\mathbf{P}^1_k/\Delta)(k,M) \subseteq \mathbf{P}^1_k$, where $M \subseteq \Omega_k$ is a (sufficiently big) finite set of finite places of $k$ and $\Delta$ is the \emph{orbifold base} \cite[\S 1.2]{campana} of the fibration. In Campana's example, $\Delta$ can be taken to be a $\mathbf{Q}$-divisor of the form $\sum_{i = 1}^m \frac12 P_i$, where $m \geq 5$ and $P_i \in \mathbf{P}^1(k)$ for $1 \leq i \leq m$. 

The crucial point is that this orbifold is \emph{of general type} \cite[\S 1.1 \& Remarque 1.6]{campana}. Campana conjectured in \cite[Conjecture 4.5]{campana} (the ``orbifold Mordell conjecture'') that the set $(\mathbf{P}^1_k/\Delta)(k,M)$ is then finite for any number field $k$ and any finite set of places $M \subseteq \Omega_k$. Assuming this conjecture, Proposition \ref{chatelet} now immediately implies the existence of a family of Ch\^atelet surfaces $X$ over $S$ such that $X(\mathbf{A}_k)^{\text{\'et,Br}} \neq \emptyset$ and $X(k) = \emptyset$. By Remark \ref{pi1}, the fourfold $X$ is simply connected.

Finally, it turns out that Campana's conjecture \cite[Conjecture 4.5]{campana} is a consequence of the $abc$ conjecture; this is proven in a special case in \cite[Remarque 4.8]{campana}, but this case does not cover the one we need here. Abramovich \cite[\S 4.4]{abramovich} suggests using Belyi maps and adapting the proof of Elkies' result \cite{elkies} to prove the general case. Following his suggestions, we will include a complete proof in the appendix, in order to fill this gap in the literature. \end{proof}

\section{Appendix: $abc$ implies ``orbifold Mordell''}

Let $k$ be an arbitrary field. Define an \emph{orbifold curve} $(C/\Delta)$ over $k$ to be a smooth, proper, geometrically connected curve $C$ over $k$, together with a $\mathbf{Q}$-divisor $\Delta$ of the form $$\Delta = \sum_{1 \leq j \leq N} \left(1 - \frac1{m_j}\right)\Delta_j,$$ where $m_j \geq 2$ is an integer and $\Delta_j$ is a prime divisor on $C$ (defined over $k$) for $1 \leq j \leq N$.

The $\mathbf{Q}$-divisor $K_C + \Delta$ is the \emph{canonical divisor} of $(C/\Delta)$; the \emph{Kodaira dimension} of $(C/\Delta)$ is defined as the Iitaka dimension $\kappa(C/\Delta) := \kappa(C,K_C + \Delta)$. An orbifold curve $(C/\Delta)$ is said to be \emph{of general type} \cite[\S 1.1]{campana} if $\kappa(C/\Delta) = 1$; this is the case exactly when the inequality \begin{equation} \label{general} 2g(C) - 2 + \sum_{1 \leq j \leq N} \left(1 - \frac{1}{m_j}\right) \deg \Delta_j > 0 \end{equation} holds, where $g(C)$ is the genus of $C$, as can be easily checked using \cite[Remark 1.6]{campana}.

Campana introduced the notion of \emph{orbifold rational points} for these objects when $k$ is a number field. Let $\mathcal{O}_k$ be the ring of integers of $k$. Fix a model of $(C/\Delta)$ defined over $\mathcal{O}_{k,S}$ for some finite set $S$ of places of $k$; we assume that $S$ contains the places where $C$ has bad reduction. For each $q \in C(k)$ and for each finite place $v$ of $k$, we define the arithmetic intersection number $(q \cdot \Delta_j)_v$ of $q$ and $\Delta_j$ as the biggest integer $m \geq 0$ such that $q$ and $\Delta_j$ have the same image in the reduction of $C$ modulo $p_v^m$ (this depends on the model chosen). The set of orbifold rational points $(C/\Delta)(k,S)$ (which again depends on the model) is then defined as the set of rational points $q \in C(k)$ such that for each finite place $v \not\in S$ and $1 \leq j \leq N$, we have either $(q \cdot \Delta_j)_v = 0$ or $(q \cdot \Delta_j)_v \geq m_j$. 

Campana \cite[Conjecture 4.5]{campana} conjectured the following result (not in this precise form):

\begin{conj} Let $k$ be a number field, let $S$ be a finite set of places of $k$ and let $(C/\Delta)$ be an orbifold curve over $k$ which is of general type. Then $(C/\Delta)(k,S)$ is finite. \end{conj}

We will prove that this follows from the $abc$ conjecture by adapting Elkies' proof \cite{elkies} of the fact that the $abc$ conjecture implies the ``classical'' Mordell conjecture, i.e. Faltings' theorem. 

Let us first introduce some notation. Given $a,b,c \in k^\star$, define the \emph{height} $$H(a,b,c) = \prod_{v}  \max(\left\|a\right\|_v, \left\|b\right\|_v, \left\|c\right\|_v)$$ where the product ranges over all (normalized) places of $k$, and the \emph{conductor} $$N(a,b,c) = \prod_{v \in I} N(v),$$ where the product ranges over the set of all finite places $v$ for which $$\max(\left\|a\right\|_v, \left\|b\right\|_v, \left\|c\right\|_v) > \min(\left\|a\right\|_v, \left\|b\right\|_v, \left\|c\right\|_v)$$ and $N(v)$ denotes the (absolute) norm of the corresponding prime ideal. 

\begin{conj}[$abc$ conjecture] If $a,b,c \in k^\star$ such that $a + b + c = 0$, then $$N(a,b,c) \gg_{\varepsilon, k} H(a,b,c)^{1 - \varepsilon}$$ for each $\varepsilon > 0$, where the implied constant depends only on $\varepsilon$ and the number field $k$.  \end{conj}

Since $H(a,b,c)$ and $N(a,b,c)$ are both scaling invariant, they depend only on the ratio $$r = -a/b \in \mathbf{P}^1(k) \setminus \{0,1,\infty\}.$$ The number $N(a,b,c)$ is precisely the product of the absolute norms of all finite primes of $K$ at which $r$, $r - 1$ or $1/r$ has positive valuation; let $N_0(r)$, $N_1(r)$ and $N_\infty(r)$ be the corresponding products, so that $$N(r) = N_0(r) N_1(r) N_\infty(r).$$

We will now prove, following Elkies,

\begin{theo} Assume that the $abc$ conjecture holds; then Campana's conjecture holds, i.e. if $k$ is a number field, $S$ a finite set of places of $k$ and $(C/\Delta)$ an orbifold curve over $k$  of general type, then the set $(C/\Delta)(k,S)$ is finite. \end{theo}

\begin{proof} As suggested by Abramovich, let $f: C \to \mathbf{P}^1_k$ be a \emph{Belyi map} ramified only above $0$, $1$ and $\infty$ such that $f^{-1}(\{0,1,\infty\})$ contains the support of $\Delta$. Define a $\mathbf{Q}$-divisor on $C$ by $$D_0 = \sum_{f(\delta) = 0} \alpha_\delta \delta,$$ where the sum runs over all prime divisors $\delta$ on $C$ (defined over $k$) contained in $f^{-1}(0)$, $\alpha_\delta = \frac{1}{m_j}$ if $\delta = \Delta_j$ for some $j$, and $\alpha_\delta = 1$ if not. Define $D_1$ and $D_\infty$ similarly. 

Define $D = D_0 + D_1 + D_\infty$. Let $b_f(\delta)$ be the branch number of $f$ at $\delta$; then $$\deg D = 3 \deg f - \sum_{f(\delta) \in \{0,1,\infty\}} b_f(\delta) \deg \delta - \sum_{1 \leq j \leq N} \left(1 - \frac{1}{m_j}\right) \deg \Delta_j.$$ By Riemann-Hurwitz, we have $$\sum_{f(\delta) \in \{0,1,\infty\}} b_f(\delta) \deg \delta = 2 \deg f + 2g(C) - 2$$ and hence $$\deg D = \deg f - 2g(C) + 2 - \sum_{1 \leq j \leq N} \left(1 - \frac{1}{m_j}\right)  \deg \Delta_j < \deg f$$ since $(C/\Delta)$ is of general type (see (\ref{general})).

It will now be sufficient to prove the inequality \begin{equation}\label{key}\log N_0(f(p)) < \frac{\deg D_0}{\deg f} \log H(f(p)) + O\left(\sqrt{\log H(f(p))} + 1\right)\end{equation} for any \emph{orbifold} rational point $p \in C(k) \setminus f^{-1}(0)$. Indeed, together with similar equalities for $N_1(f(p))$ and $N_\infty(f(p))$ this would imply that for any $p \in C(k) \setminus f^{-1}(\{0,1,\infty\})$ which is an orbifold rational point, we have the following inequality: $$\log N(f(p)) < \frac{\deg D}{\deg f} \log H(f(p)) + O\left(\sqrt{\log H(f(p))} + 1\right),$$ with the implied $O$-constant effective and not depending on $p$. This would mean that $f(p)$ gives a counterexample to the $abc$ conjecture as soon as the inequality $$\varepsilon > 1 - \frac{\deg D}{\deg f}$$ holds and $H(f(p))$ is sufficiently large, i.e. for all but finitely many $p$.

Let us therefore focus on proving (\ref{key}). Following Elkies, let us fix for any divisor $c$ on $C$ a logarithmic height function $h_c$ relative to $c$, which is well defined up to $O(1)$. Let us write $\overline{D}$ for the zero divisor of $f$; write $\overline{D} = \sum_{f(\delta) = 0} n_\delta \delta$ where the $\delta$ are distinct prime  divisors of multiplicity $n_\delta$. Hence $$\deg f = \sum_{f(\delta) = 0} n_\delta \deg \delta.$$ We now have $$\log H(f(p)) = h_{\overline{D}}(p) + O(1) = \sum_{f(\delta) = 0} n_\delta h_{\delta}(p) + O(1).$$ Outside of a finite set $S$ of primes of $k$, a prime contributes to $N_0(f(p))$ if and only if it contributes to $h_{\delta}(p)$ for some $\delta \in f^{-1}(0)$. However, if such a prime $v$ contributes to $N_0(f(p))$ and to $h_{\Delta_j}(p)$ for some $j$ with $1 \leq j \leq N$, then its contribution to the value of $N_0(f(p))$ is exactly $N(v)$, whereas its contribution to the value of $h_{\Delta_j}(p)$ is at least $m_j N(v)$, by the definition of an ``orbifold rational point''. Since the contribution of any (archimedean) prime in $S$ to the height relative to an effective divisor is bounded below, we get $$\log N_0(f(p)) < \sum_{f(\delta) = 0} \alpha_\delta h_{\delta}(p) + O(1) = h_{D_0}(p) + O(1),$$ where (as before) $\alpha_\delta =\frac{1}{m_j}$ if $\delta = \Delta_j$ is in the orbifold locus and $\alpha_\delta = 1$ if not.  

One can now finish the argument like in Elkies' paper: it remains to prove that $$h_{D_0}(p) = \frac{\deg D_0}{\deg f} \, h_{\overline{D}}(p) + O(\sqrt{\log H(f(p))} + 1)$$ which is equivalent to $$h_{\Delta}(p) = O(\sqrt{\log H(f(p))} + 1)$$ where $\Delta$ is the divisor of degree zero on $C$ given by $$\Delta = \deg \overline{D} \cdot D_0 - \deg D_0 \cdot \overline{D}.$$ This follows from a theorem of N\'eron \cite[p. 45]{serre}. \end{proof}

\end{document}